\RequirePackage{ifpdf}
\ifpdf 
\documentclass[pdftex]{sigma}
\else
\documentclass{sigma}
\fi

\usepackage[all,ps,cmtip]{xy}

\begin{document}

\allowdisplaybreaks

\renewcommand{\thefootnote}{$\star$}

\renewcommand{\PaperNumber}{078}

\FirstPageHeading

\ShortArticleName{On Spinor Varieties and Their Secants}

\ArticleName{On Spinor Varieties and Their Secants\footnote{This paper is a
contribution to the Special Issue ``\'Elie Cartan and Dif\/ferential Geometry''. The
full collection is available at
\href{http://www.emis.de/journals/SIGMA/Cartan.html}{http://www.emis.de/journals/SIGMA/Cartan.html}}}

\Author{Laurent MANIVEL}

\AuthorNameForHeading{L. Manivel}

\Address{Institut Fourier,
Universit\'e de Grenoble I et CNRS,
BP 74, 38402 Saint-Martin d'H\`eres, France}
\Email{\href{mailto:Laurent.Manivel@ujf-grenoble.fr}{Laurent.Manivel@ujf-grenoble.fr}}
\URLaddress{\url{http://www-fourier.ujf-grenoble.fr/~manivel/}}

\ArticleDates{Received April 03, 2009, in f\/inal form July 21, 2009;  Published online July 24, 2009}

\Abstract{We study the secant variety of the spinor variety, focusing on its equations
of degree three and four. We show that in type $D_n$, cubic equations exist
if and only if $n\ge 9$. In general the ideal has generators in degrees
at least three and four. Finally we observe that the other Freudenthal
varieties exhibit strikingly similar behaviors.}


\Classification{14M17; 15A66; 14L35; 14N15}

\def\ot{\otimes}
\def\we{\wedge}
\def\wec{\wedge\cdots\wedge}
\def\op{\oplus}
\def\ra{\rightarrow}
\def\lra{\longrightarrow}
\def\cO{\mathcal O}
\def\fso{\mathfrak so}
\def\fsl{\mathfrak sl}
\def\PP{\mathbb P}\def\CC{\mathbb C}
\def\RR{\mathbb R}\def\HH{\mathbb H}\def\OO{\mathbb O}
\def\smc{\cdots }

\section{Introduction}

This paper grew out from a striking observation in a preprint version
of \cite{LW2}, according to
which the spinor variety of type $D_7$ was the unique compact Hermitian
symmetric space whose secant variety (with respect to its minimal equivariant
embedding) had no cubic equation.

Among the most familiar Hermitian symmetric spaces, the Scorza varieties,
which parametrize matrices (possibly symmetric or skew-symmetric) of rank one
(up to scalar), have extremely well-behaved secant varieties: the $k$-th
secant $\sigma_k(X)$ of a Scorza variety $X$ parametrizes matrices of rank
at most $k+1$, and its ideal is generated by explicit equations of
degree $k+2$, given by minors and their generalizations (a standard reference
for secant varieties, including Scorza varieties, is \cite{zak}).
In particular the f\/irst
secant variety $\sigma(X)$ is cut-out by cubics, and it was tempting to think
that $\sigma(X)$, for many equivariantly embedded rational homogeneous spaces,
or at least for the simplest ones, should have many cubic equations. Moreover,
writing down such equations explicitly would be a step towards a generalized
theory of minors.

The class of Hermitian symmetric spaces (possibly enlarged to the
so-called cominuscule varieties) being particularly well-behaved,
it was tempting to try to understand better the observation of \cite{LW2}
about the spinor variety of type $D_7$. Our f\/irst aim was to understand,
for a spinor variety $S$ of type $D_n$, the cubic equations of $\sigma(S)$.
The answer to this problem is essentially given by Theorem~\ref{cubic-dec},
where we provide the decomposition into irreducibles of the symmetric cube
of a half-spin representation, combined with Theorem~\ref{cubics}.
Our conclusion is that there are
no cubic equations for the secant variety exactly in types $D_7$ and $D_8$.
In type  $D_9$ there exists a family a cubic equations parametrized by the
other half-spin representation.

On the other hand quartic equations exist already in type $D_7$ (in type
$D_6$ and lower, the secant variety is the whole ambient space). Theorem
\ref{quartics} allows a partial  understanding of  quartic equations of
$\sigma(S)$. In particular, in any type $D_n$, with $n\ge 7$, there exist
quartic equations which cannot be derived from cubic ones.

Finally, in the last sections we put our results in a broader perspective.
The point is that spinor varieties are only one class among the so-called
Freudenthal varieties, which also include Lagrangian Grassmannians, and
Grassmannians of middle-dimensional subspaces of an even dimensional
vector space. It turns out that at least with respect to their cubic
and quartic equations, Freudenthal varieties exhibit a strikingly
uniform behavior. It would be interesting to understand better these
similarities and their limits.

\section{Spinor varieties and Pfaf\/f\/ians}

It was already observed by Chevalley that an open subset of the
spinor variety can be paramet\-ri\-zed by Pfaf\/f\/ians; more precisely,
by the complete set of sub-Pfaf\/f\/ians of a generic skew-symmetric
matrix. An interesting consequence is that the equations of the
spinor variety provide polynomial relations between these
sub-Pfaf\/f\/ians. Such relations have been known from the very
beginning of Pfaf\/f\/ian theory, the most famous one being the
rule analogous to the line or column expansion of determinants.
Many other such relations have been found, including Tanner's
relations, and its relatively recent generalization by Wenzel
(\cite{dw}, see also \cite{knuth} for a short proof).

The link  with spinor varieties has the following immediate
consequences. First, since the ideal of the spinor variety
is generated in degree two by Kostant's theorem, it follows
that all the algebraic relations between sub-Pfaf\/f\/ians are generated
by quadratic relations. Second, these quadratic relations can be
completely explicited, using the full set of quadratic equations
for the spinor variety obtained by Chevalley. Since these
equations are more convenient to express in terms of Clif\/ford
algebras, we will begin with a brief review of the tools involved.
Our basic reference is \cite{chevalley}.

\subsection{Clif\/ford algebras and spin representations}

Let $V$ be a complex vector space of dimension $2n$, endowed with a
non degenerate quadratic form $q$. Suppose that $V=E\op F$
splits as the direct sum of two maximal isotropic vector
spaces, of dimension $n$. For any $v\in V$, we can def\/ine
two natural operators on the exterior algebra~$\wedge V$.
One is the exterior product $o(v)$ by $v$; the other one is
the inner product $i(v)$ def\/ined by the contraction with $v$
through the quadratic form; both are graded operators, of
respective degree~$+1$ and~$-1$.

Let $\psi(v)=o(v)+i(v)$. An easy computation shows that the
map $\psi : V\ra {\rm End}(\wedge V)$ extends to the Clif\/ford algebra
$Cl(V)$, the quotient of the tensor algebra of $V$ by the relations
\[
v\ot w+w\ot v=2q(v,w)1.
\]
 Applying to $1\in\wedge^0V$, one obtains  a
vector space isomorphism
\[
\theta : \ Cl(V)\simeq \wedge V
\]
which is compatible
with the action of $V$ on both sides. In particular the inverse
isomorphism sends $v\wedge w\in \wedge^2 V$ to $vw-q(v,w)1\in Cl(V)$.
Note that $\wedge^2 V$ is naturally isomorphic to $\fso(V)$ and that
the resulting map $\fso(V)\ra Cl(V)$ is a morphism of Lie algebras, once
the associative algebra $Cl(V)$ has been endowed with its natural Lie
algebra structure.

Let $f\in Cl(V)$ be the product of the vectors in some basis of $F$
($f$ is well def\/ined up to a non zero scalar). Then $Cl(V)f$ is a left ideal
of $Cl(V)$ isomorphic, as a vector space, to $\wedge E$ (this is
because $Cl(V)\simeq Cl(E)\otimes Cl(F)$ and $Cl(F)f=\langle f\rangle$).
This makes of $\wedge E$ a $Cl(V)$-module, with an action of $v=v'+v''\in
V=E\oplus F$ def\/ined by $o(v')+2i(v'')$ (beware to the factor $2$ here!).

The spin group ${\rm Spin}(V)$ is def\/ined as a subgroup of the group of invertible
elements in the Clif\/ford algebra. In particular any $Cl(V)$-module
is also a ${\rm Spin}(V)$-module.
As a ${\rm Spin}(V)$-module, $\wedge E$ splits as the direct sum of the two
half-spin representations $\wedge^{\rm odd}E$ and $\wedge^{\rm even}E$,
both of dimension $2^{n-1}$, that we will also denote
$\Delta_+$ and  $\Delta_-$. A priori these two representations cannot be
distinguished.
Our convention will be the following: if $\epsilon(n)$ is the parity of~$n$,
then we let
\[
\Delta_+=\wedge^{\epsilon(n)} E.
\]
This will have the advantage that the
pure spinor def\/ined by $E$ will belong to the projectivization of
$\Delta_+$.

\subsection{Pure spinors}\label{section2.2}

The variety of maximal isotropic subspaces of $V$ has two connected
components $S_+$ and $S_-$, which are (non canonically) isomorphic.
Their linear spans in the Pl\"ucker embedding are in direct sum: there
is a splitting
\[
\wedge^nV =\wedge^nV_+ \op \wedge^nV_-
\]
into spaces of the same dimension, and $S_\pm=G(n,2n)\cap\PP(\wedge^nV_\pm)$
are linear sections of the usual Grassmannian. But the minimal embeddings
of $S_+$ and $S_-$ are embeddings in the projectivized half-spin
representations: $S_+$ is the ${\rm Spin}(V)$-orbit, in the projectivization
of $\Delta_+$, of the line $\wedge^{\rm top}E=\langle e_1\wedge\cdots\wedge e_n
\rangle$, where $e_1,\ldots ,e_n$ is a basis of $E$.
The embedded varieties
\[
S_+\subset \PP \Delta_+ \qquad \mathrm{and}\qquad
S_-\subset \PP \Delta_-
\]
are called the varieties of even and odd pure spinors, respectively. The more familiar embeddings into
$\PP(\wedge^nV_\pm)$ can be recovered from these minimal embeddings through a quadratic Veronese map.

Two maximal isotropic subspaces of $V$ are in the same
spinor variety if and only if
the dimension of their intersection has the same parity as $n$. In particular $E$ and $F$
belong to the same family if and only if $n$ is even.

If $H\subset V$ is any maximal isotropic subspace, its representative in $S_+\cup S_-$ can be obtained as follows.
Let $h\in Cl(V)$ be the product of the vectors of  a basis of $H$. Then the left ideal $Cl(V)f$ intersects the
right ideal $hCl(V)$ along a line, which can be written as $u_Hf$ for a unique line $u_H\in\PP(\wedge E)$.
According to the parity of the dimension of $H\cap E$ one gets in fact a line $u_H\in\PP(\wedge^\pm E)
=\PP\Delta_\pm$, representing the  point of the spinor variety def\/ined by $H$.

Consider for example, for any subset $I$ of $\{1,\ldots, n\}$,  the maximal isotropic space $H_I$ generated
by the vectors $e_i$ for $i\in I$, and $f_j$ for $j\notin I$, where $f_1,\ldots ,f_n$ is the basis of $F$ such that $q(e_i,f_j)=\delta_{ij}$.
Then  $u_{H_I}$ is the line generated by $e_I$,
the wedge product of the $e_i$'s, $i\in I$.

\begin{proposition}\label{3-tr}
The action of ${\rm Spin}(V)$ on $S_+$ is generically $3$-transitive $($but not $4$-transitive$)$,
in the sense that ${\rm Spin}(V)$ has an open orbit
in $S_+\times S_+\times S_+$.
\end{proposition}

\begin{proof} We treat the case where $n=2m$ is even. Then $E$ and $F$ both belong to $S_+$, as well as the
space $G=\langle e_{2i}+f_{2i-1}, e_{2i-1}-f_{2i}, 1\le i\le m\rangle$. The stabilizer of $(E,F)$ in
$S_+\times S_+$ is easily seen to be equal to $GL_n$, embedded almost diagonally in $SO_{2n}$ by
the morphism $A\mapsto (A,{ }^tA^{-1})\in GL(E)\times GL(F)\subset GL(E\oplus F)$. The additional
condition that $G$ be preserved is then equivalent to the condition that $A$ preserved the symplectic
form $\omega$ on $E$ def\/ined by $\omega(e_{2i} ,e_{2i-1})=1$ and $\omega(e_j,e_k)=0$ if ${j,k}$ is not
of the form ${2i,2i-1}$. So the stabilizer of the triple $(E,F,G)\in S_+\times S_+\times S_+$ is a copy
of the symplectic group $Sp_n$ inside $SO_{2n}$. In particular the orbit of $(E,F,G)$ inside
$S_+\times S_+\times S_+$ has dimension $\dim SO_{2n}-\dim Sp_n = 3\dim S_+$, so it must be dense.
\end{proof}

The smooth varieties $S_+$, $S_+\times S_+$, $S_+\times S_+\times S_+$ are thus smooth
compactif\/ications of the homogeneous spaces $SO_n$,  $SO_{2n}/GL_n$
and $SO_{2n}/Sp_n$, respectively.

We thus recover the fact (which holds for any rational homogeneous variety,
equivariantly embedded) that the secant variety $\sigma(S_+)$ to the spinor
variety is quasi-homogeneous. However the second secant variety
$\sigma_2(S_+)$ is not, because the stabilizer $Sp_n$ of a general triple
in $S_+$ acts trivially on the corresponding lines of $\Delta_+$. So a general
point of $\sigma_2(S_+)$ can be put, up to the group action, in the form
$t_Eu_E+t_Fu_F+t_Gu_G$, but the triple $(t_E,t_F,t_G)\in\PP^2$ cannot be
reduced to $(1,1,1)$.

  For future use we choose the following general points of
$S_+\times S_+$ and $S_+\times S_+\times S_+$:

  If $n=2m$ is even, $E$ and $F$ belong to $S_+$, and $G$ def\/ined
above as well. Their representatives in $\PP\Delta_+$ are
\[
u_E=e_1\wedge\cdots\wedge e_n, \qquad u_F=1, \qquad
u_G=\wedge_{i=1}^m(1+e_{2i-1}\wedge e_{2i}).
\]

  If $n=2m+1$ is odd, $E$ belongs to $S_+$ but $F$ belongs to $S_-$,
so we denote by $F'$ the maximal isotropic subspace generated by $f_1,
\ldots ,f_{n-1},e_n$, which belongs to $S_+$. Then
\[
u_E=e_1\wedge\cdots\wedge e_n, \qquad u_{F'}=e_n.
\]

  The sum $u_E+u_F$ (respectively $u_E+u_{F'}$) def\/ines a
generic  point of the secant variety $\sigma(S_+)$, more precisely a point
in the open orbit.

\subsection{Pfaf\/f\/ians}
Pfaf\/f\/ians appear when one tries to parametrize
the spinor variety $S_+$, at least in a neighborhood of the point
def\/ined by $E$. This is exactly similar to the fact that minors of
a generic matrix are Pl\"ucker coordinates of a generic point of
the Grassmannian.

For this, we use our preferred basis $e_1,\ldots ,e_n$ of $E$. Let $u=(u_{ij})$
be any skew-symmetric matrix of size $n$. Then the vectors
\[
e_i(u)=e_i+\sum_{j=1}^nu_{ij}f_j, \qquad 1\le i\le n
\]
generate a maximal isotropic subspace $E(u)$ in the same family as $E$,
that is $S_+$.
A computation yields the following formula for the
corresponding pure spinor (this is a slight extension of the way
Chevalley def\/ines the Pfaf\/f\/ian in \cite[p.~57]{chevalley}):
\[
u_{E(u)}=\sum_{\ell(K)\;{\rm even}}{\rm Pf}_K(u)e_{K^c},
\]
where the sum is over the sequences $K=(k_1\smc k_\ell)$ of integers
between $1$ and $n$, of even length $\ell=\ell(K)$, $K^c$ is the
complementary subset of integers, and ${\rm Pf}_K(u)$ is
the Pfaf\/f\/ian of the submatrix of $u$ obtained by taking lines and columns
indexed by $K$. Since $E(u)$ is a~generic maximal isotropic subspace in the
same family as $E$, this formula provides a rational parametrization of
a dense open subset  of the spinor variety $S_+$.

\subsection{Equations}
Which equations characterize pure spinors? Kostant's theorem asserts that,
as for any equivariantly embedded rational homogeneous space, the ideal
of $S_+$ is generated in degree two. Moreover, the quadratic equations
of $S_+$ can be written down very explicitly, as follows.

We need to introduce the main anti-automorphism of $Cl(V)$, which is
characterized by the fact that
$\alpha (v_1\cdots v_r)=v_r\cdots v_1$ if the space generated
by $v_1,\ldots, v_r$ is isotropic.
Then, for any $u,v\in \Delta_+\subset Cl(V)$, let
\[
\beta(u,v)= uf\alpha(v)\in Cl(V)\simeq\wedge V,
\]
where the isomorphism between $Cl(V)$ and $\wedge V$ is def\/ined as
above. Denote by $\beta_k(u,v)$ the projection of $\beta(u,v)$ to
$\wedge^k V$. Then
\begin{enumerate}\itemsep=0pt
\item $\beta_k$ is non zero if and only if $n-k=2p$ is even;
\item then it is symmetric for $p$ even and skew-symmetric for $p$ odd;
\item $\beta_k$ and $\beta_{2n-k}$ coincide up to sign;
\item $u\in \Delta_+$ def\/ines a pure spinor if and only if $\beta_k(u,u)=0$
for all $k$ smaller than $n$.
\end{enumerate}
Indeed, in case $u$ represents a pure spinor,
that is, a maximal isotropic subspace
of $V$, one can check that $\beta(u,u)$ is the
product of the vectors in a basis of that space. This implies that
$\beta_k(u,u)=0$ for all $k\ne n$. Conversely, the last assertion is that
this property def\/ines $S_+$, and more precisely, we get the complete
space of quadratic equations of $S_+$ -- whence the whole ideal.

A simple computation yields the following formula, which will be of
fundamental importance in the sequel. Recall that we denoted by $F'$
the maximal isotropic subspace generated by $f_1,
\ldots ,f_{n-1},e_n$.

\begin{proposition}
Suppose $n$ is even. Then, up to a non zero constant,
\[
\beta_{2k}(u_E,u_F)=\sum_{\ell(K)=k}e_K\wedge f_K.
\]
Suppose $n$ is odd. Then, up to a non zero constant,
\[
\beta_{2k+1}(u_E,u_{F'})=\sum_{\ell(K)=k}e_K\wedge f_K\wedge e_n.
\]
\end{proposition}

\subsection{Pfaf\/f\/ian identities}
Applying the previous identities to the generic maximal isotropic
subspace $E(u)$ and the corresponding pure spinor,
we shall obtain quadratic identities between
sub-Pfaf\/f\/ians of a generic skew-symmetric matrix. As explained above,
we will obtain all the relations of that kind, and every algebraic
relation between these sub-Pfaf\/f\/ians can be deduced from these.

The only computation we need is that of $\beta(e_I,e_J)=
\psi(e_If \alpha(e_J)).1$.
Let $f_1,\ldots ,f_n$ be a basis of $F$ such that $q(e_i,f_j)=\delta_{ij}$.
We can suppose that $f=f_1\cdots f_n$.

\begin{lemma}
For any $I$, $J$, we have
\[
\beta(e_I,e_J)=2^{|I\cap J|}\sum_{R\subset I\Delta J}
\epsilon(I,J,R)e_{I\cap J}\wedge e_R\wedge f_{I^c\cap J^c}\wedge f_R,
\]
with $\epsilon(I,J,R)=\pm 1$. $($We denoted $I\Delta J=(I\setminus J)\cup (J\setminus I)$ and
$I^c$ the complement of $I.)$
\end{lemma}

\begin{proof}
First we check the following
simple formula:
\[
\psi(f)\psi(e_n\ldots e_1).1=\psi(f).e_n\wec e_1=\prod_{i=1}^n
(1+f_i\wedge e_i).
\]
More generally, $\psi(f_I)\psi(e_{\bar I}).1=\prod_{i\in I}
(1+f_i\wedge e_i)$, if $\bar I$ denotes the sequence $I$
in reverse order.

Then we make the following series of computations, f\/irst without taking care
of signs:
\begin{gather*}
\beta(e_I,e_J) = \pm\psi(e_I)\psi(f_{J^c})\psi(f_J).e_{\bar J}   = \pm \psi(e_I)\psi(f_{J^c}).\prod_{j\in J}(1+f_j\we e_j) \\
\phantom{\beta(e_I,e_J)}{} = \pm \psi(e_I).\Big(f_{J^c}\we \prod_{j\in J}(1+f_j\we e_j)\Big)   = \pm \psi(e_{I\setminus J})\psi(e_{I\cap J}).\Big(f_{J^c}
\we \prod_{j\in J}(1+f_j\we e_j)\Big) \\
\phantom{\beta(e_I,e_J)}{} = \pm 2^{|I\cap J|}\psi(e_{I\setminus J}).\Big(e_{I\cap J}\we f_{J^c}
\we \prod_{j\in J\setminus I}(1+f_j\we e_j)\Big) \\
\phantom{\beta(e_I,e_J)}{} = \pm 2^{|I\cap J|}\psi(e_{I\setminus J}).\Big(f_{I\setminus J}\we f_{I^c\cap J^c}\we e_{I\cap J}
\we \prod_{j\in J\setminus I}(1+f_j\we e_j)\Big) \\
\phantom{\beta(e_I,e_J)}{} = \pm 2^{|I\cap J|}\prod_{i\in I\setminus J}(1+e_i\we f_i)\we
f_{I^c\cap J^c}\we e_{I\cap J}\we \prod_{j\in J\setminus I}(1+f_j\we e_j).
\end{gather*}
Note that the factor two appear because $\psi(e_j).(1+f_j\we e_j)=2e_j$.
Expanding the two products, one easily gets the result, up to sign.

What is the correct sign in the preceding formula? The f\/irst line
contributes by the sign $\sigma(J^c,J)$, the signature of the permutation
that puts the sequence $(J^c,J)$ in increasing order. Then the fourth
line contributes by the sign $\sigma(I\setminus J,I\cap J)$. The sixth line
contributes by the sign $\sigma(I\setminus J,I^c\cap J^c)$, and an addition
minus one to the power $|I\cap J|\times |J^c|$ because of the permutation
of the $e$ and $f$ terms. Finally, in the last lines one needs to put
$I\setminus J$ in reverse order, which contributes minus one to the power
$\binom{|I\setminus J|}{2}$. \end{proof}

  Putting things together, we obtain our quadratic equations:

\begin{theorem} For any disjoint sets of integers $R$, $S$, $T$, with $R$
and $T$ of different sizes, we have
\[
\sum_{\substack{A\cup B=R\cup T\\ A\cap B=\varnothing}}
\epsilon(S\cup A,S\cup B,R,T){\rm Pf}_{S\cup A}(u){\rm Pf}_{S\cup B}(u)=0,
\]
where $\epsilon(S\cup A,S\cup B,R,T)=\pm 1$. Moreover all the quadratic
relations between the sub-Pfaffians of a generic skew-symmetric
matrix are linear combinations of these.
\end{theorem}

The sign function $\epsilon(S\cup A,S\cup B,R,T)$ can be written
down explicitly but seemingly not in a~pleasant way.
These relations are probably known, but we do not know any suitable reference.

\begin{remark} To be precise, the relations of the theorem will be non
trivial only if the sizes of~$T$ and~$R$ are such that
$|T|=|R|+4p$ for some $p>0$, as follows from the properties of~$\beta$.
If we take this restriction into account, the theorem
provides a complete set of independent
quadratic relations between sub-Pfaf\/f\/ians.

In particular the relations found by Wenzel \cite{dw,knuth} have
to follow from these.
What about the contrary? It is easy to see that Wenzel's
relations are not independent.
But it is quite possible that they generate the
full set of quadratic relations.
\end{remark}

\section{Decomposition formulas: cubics}

\subsection{Tensor products of fundamental representations}

The map $\beta$ allows to establish the following decomposition formulas
\[
S^2\Delta_\pm = \wedge^nV_\pm\oplus\bigoplus_{j>0}\wedge^{n-4j}V,
\qquad \wedge^2\Delta_\pm = \bigoplus_{j>0}\wedge^{n-4j+2}V.
\]
In the f\/irst formula $\wedge^nV_+$ and $\wedge^nV_-$ are the irreducible modules of highest
weights $2\omega_n$ and $2\omega_{n-1}$, respectively.
Also recall that $\wedge^{n-1}V$ is an irreducible but not fundamental
module, its highest weight being $\omega_n+\omega_{n-1}$. This implies that
it appears inside the tensor product of the two half-spin representations,
which decomposes as
\[
\Delta_+\ot\Delta_-=\wedge^{n-1}V\op
\bigoplus_{j>0}V_{\omega_{n-2j-1}}.
\]
We aim to generalize these formulas to degree three.

\begin{proposition}\label{ext-spin} Let $i\le n-2$; then
\[
\Delta_+\ot \wedge^iV=\bigoplus_{j\ge 0}V_{\omega_n+\omega_{i-2j}}
\oplus \bigoplus_{j\ge 0}V_{\omega_{n-1}+\omega_{i-2j-1}},
\]
where the first $($resp.\ second$)$ sum is over the set of non negative
integers $j$ such that $i-2j\ge 0$ $($resp.\ $i-2j+1\ge 0)$, and we use
the convention that $\omega_0=0$.
\end{proposition}

\begin{proof} We f\/irst produce a non zero equivariant map from
$\Delta_+\ot \wedge^iV$ to each $V_{\omega_n+\omega_{i-2j}}$
or $V_{\omega_{n-1}+\omega_{i-2j-1}}$,
which will imply that the left hand side of the identity contains
the right hand side. Then we prove the equality by checking dimensions.

For the f\/irst step, we observe that the action of $V$ on the spin
representations induces equivariant maps $\Delta_\pm\ot \wedge^{2j}V
\ra \Delta_\pm$ and $\Delta_\pm\ot \wedge^{2j+1}V
\ra \Delta_\mp$. Dualizing, we obtain maps $\Delta_\pm
\ra \Delta_\pm\ot \wedge^{2j}V$ and $\Delta_\pm
\ra \Delta_\mp\ot \wedge^{2j+1}V$. Hence the following sequence
of equivariant morphisms,
\[
\Delta_+\ot \wedge^iV\ra \Delta_+\ot \wedge^{2j}V\ot\wedge^iV
\ra \Delta_+\ot \wedge^{i-2j}V\ra V_{\omega_n+\omega_{i-2j}},
\]
where the second arrow is a contraction map by the quadratic form,
and the last one is a Cartan product. Similarly, we can def\/ine
the sequence
\[
\Delta_+\ot \wedge^iV\ra \Delta_-\ot \wedge^{2j+1}V\ot\wedge^iV
\ra \Delta_-\ot \wedge^{i-2j-1}V\ra V_{\omega_{n-1}+\omega_{i-2j-1}}.
\]
The composed morphisms are non zero and our f\/irst claim follows.

To check that dimensions f\/it, we f\/irst note that Weyl's dimension formula
yields, for any $k\le n-2$ (including $k=0$)
\[
\dim V_{\omega_n+\omega_{k}}=\dim V_{\omega_{n-1}+\omega_{k}}=
2^{n-1}\frac{2n-2k+1}{2n-k+1}\binom{2n}{k}.
\]
The required equality is thus equivalent to the identity
\[
\binom{2n}{i}=\sum_{k=0}^i\frac{2n-2k+1}{2n-k+1}\binom{2n}{k}
=\sum_{k=0}^i\left(\binom{2n}{k}-\binom{2n}{k-1}\right),
\]
which is obvious. \end{proof}

\begin{proposition}
\begin{gather*}
\Delta_+\ot \wedge^nV_+  =  V_{3\omega_n}\oplus
\bigoplus_{j>0}V_{\omega_n+\omega_{n-2j}}, \\
\Delta_+\ot \wedge^nV_-  =  V_{\omega_n+2\omega_{n-1}}\oplus
\bigoplus_{j>0}V_{\omega_{n-1}+\omega_{n-2j-1}}.
\end{gather*}
\end{proposition}

\begin{proof}
As in the proof of the Proposition above, we f\/irst notice that there exist
non-zero equivariant morphisms from the tensor products of the left hand sides
of these decomposition formulas, to any irreducible component of their
right hand sides.

There remains to check that dimensions f\/it. Weyl's dimension formula yields
\[
\dim V_{3\omega_n} = \frac{2^{n-1}}{n+1}\binom{2n}{n}
=2^{n-1}\left(\binom{2n}{n}-\binom{2n}{n-1}\right).
\]
The equality of dimensions is thus equivalent to the identity
\[
\frac{1}{2}\binom{2n}{n}=\binom{2n}{n}-\binom{2n}{n-1}+
\binom{2n}{n-2}-\cdots +(-1)^n\binom{2n}{0},
\]
which follows immediately from the binomial expansion of $(1-1)^{2n}$.\end{proof}

In a similar way, we can prove the following identity:

\begin{proposition}
\[
\Delta_+\ot \wedge^{n-1}V  = V_{2\omega_n+\omega_{n-1}}
\oplus \bigoplus_{j>0}(V_{\omega_n+\omega_{n-2j-1}}\oplus
V_{\omega_{n-1}+\omega_{n-2j}}).
\]
\end{proposition}

Of course we can deduce the corresponding identities for $\Delta_-$
by simply exchanging $\omega_n$ and $\omega_{n-1}$ in the formulas above.

\subsection{Decomposing cubics}
Now we come to our main result, a decomposition formula for $S^3\Delta_+$.
Since this formula will not be multiplicity free, we cannot proceed as
in the preceding proofs. Instead we will use induction, and
we will need some
restriction formulas for certain $\fso_{2n}$-modules to $\fso_{2n-2}$.
We will denote by $U$ the natural $\fso_{2n-2}$-module, the fundamental
weights by $\phi_1,\ldots, \phi_{n-1}$, and by~$U_\phi$ the irreducible
$\fso_{2n-2}$-module with highest weight~$\phi$. Also we will let $\delta_\pm$
be the two half-spin representations.

It is well-known that the restrictions of $\Delta_+$ and $\Delta_-$
are both equal to  $\delta_+\op\delta_-$. More generally,
\[
{\rm Res}^{\fso_{2n}}_{\fso_{2n-2}}V_{k\omega_n}=
{\rm Res}^{\fso_{2n}}_{\fso_{2n-2}}V_{k\omega_{n-1}}
=\bigoplus_{i+j=k}U_{i\phi_{n-1}
+j\phi_{n-2}}.
\]
To state our next results we will use the following notation:
we let
\begin{gather*}
S_i   =  U_{\phi_{n-1}+\phi_i}\op U_{\phi_{n-2}+\phi_i}\qquad
\mathrm{for} \quad 0\le i\le n-3, \\
S_{n-2}  =  U_{2\phi_{n-1}+\phi_{n-2}}\op U_{\phi_{n-1}+2\phi_{n-2}}, \\
S_{n-1}  =  U_{3\phi_{n-1}}\op U_{3\phi_{n-2}}.
\end{gather*}

\begin{lemma}
One has the following restriction formulas:
\begin{gather*}
Res^{\fso_{2n}}_{\fso_{2n-2}}V_{\omega_n+\omega_i} =
 S_{i}\op 2S_{i-1}\op S_{i-3}, \qquad for \quad 0\le i\le n-2,\\
Res^{\fso_{2n}}_{\fso_{2n-2}}V_{2\omega_n+\omega_{n-1}} =
 S_{n-1}\op S_{n-2}\op S_{n-3}.
\end{gather*}
\end{lemma}

\begin{proof} Restrict our formulas for tensor products to $\fso_{2n-2}$
and use induction. \end{proof}

We can now prove the main result of this section.

\begin{theorem}\label{cubic-dec}
Let $a_j$, for $j\ge 0$,  be the coefficients of the power series
\[
\sum_{j\ge 0}a_jx^j=\frac{1}{(1-x^2)(1-x^3)}.
\]
Then the third symmetric power of a half-spin representation decomposes as
\[
S^3\Delta_+=V_{3\omega_n}\oplus
\bigoplus_{j>0}a_jV_{\omega_n+\omega_{n-2j}}\oplus
\bigoplus_{j\ge 4}a_{j-4}V_{\omega_{n-1}+\omega_{n-2j-1}}.
\]
\end{theorem}

\begin{proof} First observe that $S^3\Delta_+$ is a submodule of $S^2\Delta_+\ot
\Delta_+$. Since the decomposition of $S^2\Delta_+$ into irreducible
components only involves wedge powers of the natural representation, we easily
deduce from the formulas we already proved, that $S^3\Delta_+$ must be
a sum of modules of the form $V_{3\omega_n}$ (with multiplicity one),
and $V_{\omega_n+\omega_{n-2j}}$ or $V_{\omega_{n-1}+\omega_{n-2j-1}}$,
with $j>0$. We can thus let
\[
S^3\Delta_+=V_{3\omega_n}\oplus
\bigoplus_{j>0}a^n_jV_{\omega_n+\omega_{n-2j}}\oplus
\bigoplus_{j>0}b^n_{j}V_{\omega_{n-1}+\omega_{n-2j-1}},
\]
for some multiplicities $a^n_j$, $b^n_j$ which a priori, can depend on $n$.
We want to compute these multiplicities inductively, by restricting this
formula to $\fso_{2n-2}$. A straightforward computation yields the
following relations:
\begin{gather*}
 a^{n-1}_j+\alpha_j = 2a^n_j+b^n_j+b^n_{j-1}, \qquad
 b^{n-1}_{j-1}+\beta_{j-1} = a^n_{j}+a^n_{j-1}+2b^n_{j-1},
\end{gather*}
where for $j\ge 1$,
$\beta_j$ (resp.\ $\alpha_j$) is the number of pairs $(k,l)$
such that $k+2l=j$ and $k,l\ge 0$ (resp.\ $k\ge 0$, $l\ge 1$).
The two  equations are valid for $j=1$ if we let all the
coef\/f\/icients with index zero equal to zero. In particular
the second one yields $a^n_1=0$.

Now we can use the two
equations alternatively, to compute $b^n_j$ and $a^n_j$ by
induction on $j$. Indeed the f\/irst equation gives $b^n_j$
knowing $a^n_j$ (and coef\/f\/icients computed before), and then
the second equation gives $a^n_{j+1}$ knowing $a^n_j$ and $b^n_j$.

But since $\alpha_j$ and $\beta_j$ do not depend on $n$, and neither
does $a^n_1$, which can be seen as the input of the induction, we can
conclude that $a^n_j$ and $b^n_j$ are all independent of $n$.
So we denote them simply by  $a_j$ and $b_j$ and we let
\[
a(x)=\sum_{j>0}a_jx^j, \qquad b(x)=\sum_{j>0}b_jx^j.
\]
Our induction relations above can then be rewritten as
\begin{gather*}
a(x)+(1+x)b(x) = \frac{x^2}{(1-x)(1-x^2)}, \\
(1+x)a(x)+xb(x) = \frac{x}{(1-x)(1-x^2)}-x.
\end{gather*}
Note that the determinant of the matrix of coef\/f\/icients of $a(x)$ and $b(x)$
equals $(1+x)^2-x=1+x+x^2$. When we solve these two equations, this explains
the appearance of a factor $(1-x^3)$ at the denominators of $a(x)$ and $b(x)$.
Indeed we easily get
\[
a(x)=\frac{1}{(1-x^2)(1-x^3)}-1, \qquad b(x)=\frac{x^4}{(1-x^2)(1-x^3)}.
\]
The theorem is proved. \end{proof}

  Comparing dimensions, one gets the curious corollary:

\begin{corollary}
Let $c_p$, for $p\ge 0$,  be the coefficients of the power series
\[
\sum_{p\ge 0}c_px^p=\frac{1+x^9}{(1-x^2)(1-x^3)}.
\]
Then for all any integer $n\ge 2$, one has the identity
\[
\frac{(2^{n-1}+1)(2^{n-2}+1)}{3} = \sum_{p\ge 0}c_p
\left(\binom{2n}{n-p} -\binom{2n}{n-p-1}\right).
\]
\end{corollary}

Moreover, $c_p$, $p\ge 0$, is the only series with non negative
coef\/f\/icients satisfying this identity.

\section{Cubic equations of the secant variety}

\begin{theorem}\label{cubics}
Restrictions of cubics to the secant variety of the
spinor variety yields
\[
\CC[\sigma (S_+)]_3=V_{3\omega_n}^\vee\oplus \bigoplus_{i>1}V_{\omega_n
+\omega_{n-2i}}^\vee.
\]
\end{theorem}

(Note the appearance of duals, since $S_+\subset \PP\Delta_+$ implies that
$\CC[\sigma (S_+)]$ is a quotient of ${\rm Sym}(\Delta_+^\vee)$.)
Comparing with the full decomposition of $S^3\Delta_+^\vee$ given by Theorem~\ref{cubic-dec}, we immedia\-te\-ly
deduce the decomposition of the space $I_3(\sigma (S_+))$ of cubic equations of
the secant variety:
\[
I_3(\sigma (S_+))=\bigoplus_{j\ge 6}a_{j-6}V_{\omega_n+\omega_{n-2j}}^\vee
\oplus
\bigoplus_{j\ge 4}a_{j-4}V_{\omega_{n-1}+\omega_{n-2j-1}}^\vee.
\]
(Note that the term $a_{j-6}$ appears because of the relation $a(x)-
\frac{x^2}{1-x}=x^6(a(x)+1)$.)  In particular, we deduce:

\begin{corollary}
The secant variety of the spinor variety of type $D_n$ has non trivial
cubic equations if and only if $n\ge 9$.
\end{corollary}

In type $D_9$ we have
\[
S^3\Delta_+ = V_{3\omega_9}\oplus V_{\omega_9+\omega_5}
\oplus V_{\omega_9+\omega_3}\oplus V_{\omega_9+\omega_1}
\oplus V_{\omega_8},
\]
and $I_3(\sigma (S_+))=V_{\omega_8}^\vee.$

  We can compare Theorem~\ref{cubics}  with Theorem 3.11 in \cite{LW1}, according
to which the coordinate ring of the tangent variety $\tau(S_+)$ is given in
degree $d$ by the formula
\[
\CC[\tau (S_+)]_d=
\bigoplus_{2\sum a_p\le\sum pa_p\le d}
V_{(d-2\sum a_p)\omega_n+\sum a_p\omega_{n-2p}}^\vee,
\]
the sum being over $r$-tuples $a=(a_1,\ldots ,a_r)$ of non-negative integers,
with $r=\lfloor n/2\rfloor$ (and recall the convention that $\omega_0=0$). (Actually there is
a misprint in \cite{LW1}, where the condition stated on $a$ is not the condition
that follows from the proof.) For $d=3$ we deduce the following statement:

\begin{corollary}
The kernel of the restriction map $\CC[\sigma(S_+)]_3\ra \CC[\tau(S_+)]_3$ is
$\bigoplus_{p>3} V_{\omega_{n}+\omega_{n-2p}}^\vee.$
\end{corollary}

\begin{proof}[Proof of the theorem.]
We adopt the following strategy. First, we identify a set of generators of
the vector spaces of equivariant maps $G-{\rm Hom}(S^3\Delta_+,V_{\omega_n
+\omega_{n-2i}})$ and $G-{\rm Hom}(S^3\Delta_+$, $V_{\omega_{n-1}
+\omega_{n-2i-1}})$, for  each $i$. Then we evaluate these generators
at a general point of the secant variety and prove that their image
has rank one and zero, respectively.

 {\it First step}. By Theorem \ref{cubic-dec}, we know
that the dimension of $G-{\rm Hom}(S^3\Delta_+,V_{\omega_n +\omega_{n-2i}})$
is equal to the number of pairs $(a,b)$ of integers such that $2a+3b=i$.
Rather than constructing an explicit basis, it will be much easier to
 construct
one of $G-{\rm Hom}(S^2\Delta_+\otimes\Delta_+,V_{\omega_n +\omega_{n-2i}})$.
By restriction to $S^3\Delta_+\subset S^2\Delta_+\otimes\Delta_+$, we
will deduce a set of generators of $G-{\rm Hom}(S^3\Delta_+,
V_{\omega_n +\omega_{n-2i}})$.
Recall the decomposition
\[
S^2\Delta_+ = \wedge^nV_+\oplus\bigoplus_{j>0}\wedge^{n-4j}V.
\]
By Proposition \ref{ext-spin}, there exists an equivariant map
from $\wedge^{n-4j}V\otimes\Delta_+$ to $V_{\omega_n +\omega_{n-2i}}$
if and only if $i\ge 2j$. Moreover this map is then unique (up to constant),
and can be described as the following composition:
\begin{gather*}
\phi_{i,j} : \quad
\wedge^{n-4j}V\otimes\Delta_+ \lra \;\wedge^{n-2i}V\otimes\wedge^{2i-4j}V
\otimes\Delta_+
  \lra  \wedge^{n-2i}V\otimes\Delta_+
 \lra   V_{\omega_n +\omega_{n-2i}}.
\end{gather*}
Here the f\/irst arrow is induced by the dual map to the exterior multiplication
of exterior forms; the second one by the action of the Clif\/ford algebra
(which is isomorphic to the exterior algebra) on the half-spin representations;
 the last one is the projection to the Cartan component.

The $\phi_{i,j}$, where $i\ge 2j$, form a basis of
$G-{\rm Hom}(S^2\Delta_+\otimes\Delta_+,V_{\omega_n +\omega_{n-2i}})$.

 {\it Second step}. Now we need to evaluate
$\phi_{i,j}$ on a generic point of the secant variety, which will of course
be our favourite point introduced at the very end of Section~\ref{section2.2}.

We will consider the case were $n=2m$ is even, the case $n$ odd being similar.
Then $u_E=e_1\cdots e_n=e_{\max}$ and
$u_F=1=e_\varnothing$. We f\/irst evaluate the image
$\psi_{2i,2j}$ of $e_{\max}e_\varnothing\otimes e_\varnothing$
inside  $\wedge^{n-2i}V\otimes\Delta_+$.
\end{proof}

\begin{lemma}
Up to a non zero constant,
\[
\psi_{2i,2j}=\sum_{\substack{\ell(I)=m-i-j\\ \ell(J)=2j}}
e_I\wedge f_I\wedge f_J\otimes e_J.
\]
\end{lemma}

What remains to do is to evaluate the dimension of the subspace of
$V_{\omega_n +\omega_{n-2i}}$ spanned by the images
$\bar\psi_{2i,2j}$ of the $\psi_{2i,2j}$, $i\ge 2j$. It turns out that this
dimension is equal to one, because of the following dependence relations.

\begin{lemma}\label{relations}
 For $j\ge 1$ and $i+j\le m$, we have
\[
(-1)^{m-i-j}(2j-1)\bar\psi_{2i,2j}+(m+i-j+1)\bar\psi_{2i,2j-2}=0.
\]
\end{lemma}

\begin{proof}
First observe that the natural maps $V\otimes \Delta_\pm\ra \Delta_\mp$
 induce, by transposition, maps  $\Delta_\pm\ra\Delta_\mp\otimes V^\vee
 \simeq \Delta_\mp\otimes V$.
We can thus def\/ine an equivariant morphism
\[
\kappa_i :
\wedge^{n-2i-1}V\otimes\Delta_-\lra \wedge^{n-2i-1}V\otimes V\otimes\Delta_+
\lra \wedge^{n-2i}V\otimes\Delta_+.
\]
Observe that the image of this map is contained in (and actually coincides
with, but we will not need that) the kernel of the projection to the Cartan
component $V_{\omega_n +\omega_{n-2i}}$. Indeed, it follows from Proposition
 \ref{ext-spin} that the tensor product $\wedge^{n-2i-1}V\otimes\Delta_-$
does not contain any copy of $V_{\omega_n +\omega_{n-2i}}$.

Now consider $\psi_{2i+1,2j-1}\in \wedge^{n-2i-1}V\otimes\Delta_-$. Its image
by $\kappa_i$ is
\begin{gather*}
\kappa_i(\psi_{2i+1,2j-1}) = \sum_{k=1}^n\Bigl(
\sum_{\substack{\ell(I)=m-i-j\\ \ell(J)=2j-1}}
e_k\wedge e_I\wedge f_I\wedge f_J\otimes f_k.e_J
+f_k\wedge e_I\wedge f_I\wedge f_J\otimes e_k\wedge e_J\Bigr)\\
 \hspace{24mm}= \sum_{k=1}^n\Bigl(
\sum_{k\in J}(-1)^{\ell(I)}
e_{I\cup k}\wedge f_{I\cup k}\wedge f_{J\setminus k}\otimes e_{J\setminus k}+
\sum_{k\notin I\cup J}
e_I\wedge f_I\wedge f_{J\cup k}\otimes  e_{J\cup k}\Bigr) \\
 \hspace{24mm}= (-1)^{m-i-j}(2j-1)\psi_{2i,2j-2}+(m+i-j+1)\psi_{2i,2j}.
\end{gather*}
As we have seen the projection of $\kappa_i(\psi_{2i+1,2j-1})$
in $V_{\omega_n +\omega_{n-2i}}$ must be zero, and this implies the claim.
\end{proof}

The coef\/f\/icients in the previous dependence relations are always non zero, so
the dimension of the span of the $\bar\psi_{2i,2j}$ is at most one. What
remains to check is that it is non zero, which follows from the next lemma.

\begin{lemma}
$\bar\psi_{2i,0}\ne 0$.
\end{lemma}

\begin{proof} Recall that $\psi_{2i,0}=\sum_{\ell(I)=m-i}
e_I\wedge f_I\otimes 1\in\wedge^{n-2i}V\otimes \Delta_+$.
We need to prove that this tensor has a non zero Cartan component.
To detect the component on $V_{\omega_n+\omega_{n-2i}}$, we just need
to pair $\psi_{2i,0}$ with a highest weight vector in this representation,
that is, a tensor of the form $g_1\wedge\cdots\wedge g_{n-2i}\otimes u_G$,
where $G$ is a maximal isotropic space in the same family as $E$, and
the vectors $g_1,\ldots, g_{n-2i}$ belong to $G$. Generically, the pairing
of such a tensor with $\psi_{2i,0}$ is clearly non zero: this is obvious
for the pairing of $1$ with $u_G$ in $\Delta_+$, which is the top-degree
component of $u_G$; and also for the pairing with
$g_1\wedge\cdots\wedge g_{n-2i}$, since such tensors generate the full
 $\wedge^{n-2i}V$. \end{proof}

Now we consider cubics of type $V_{\omega_{n-1} +\omega_{n-2i-1}}$.
As before we f\/irst identify a basis
of $G-{\rm Hom}(S^2\Delta_+\otimes\Delta_+,V_{\omega_{n-1} +\omega_{n-2i-1}})$,
given by the composition $\phi_{i,j}^-$ of the following natural maps
\begin{gather*}
\phi_{i,j}^- : \quad
\wedge^{n-4j}V\otimes\Delta_+ \lra \wedge^{n-2i-1}V\otimes\wedge^{2i-4j+1}V
\otimes\Delta_+  \\
\phantom{\phi_{i,j}^- : \quad
\wedge^{n-4j}V\otimes\Delta_+}{}
 \lra \wedge^{n-2i-1}V\otimes\Delta_-
   \lra V_{\omega_{n-1} +\omega_{n-2i-1}}.
\end{gather*}
Again for $n=2m$ even, we evaluate $\phi_{i,j}^-$ at the same point as before,
and we get, up to a non zero constant, the tensor $\bar\psi_{2i+1,2j+1}$.
As before, when $j$ varies, the $\bar\psi_{2i+1,2j+1}$ are modif\/ied only by a non
zero constant. And this implies that they are all equal to zero, because

\begin{lemma}
$\bar\psi_{2i+1,1}=0$.
\end{lemma}

\begin{proof}
Indeed, the very same computation as that of Lemma~\ref{relations} shows that
$\psi_{2i+1,1}$ is a non zero multiple of $\kappa_i(\psi_{2i,0})$, and therefore
its projection to $V_{\omega_{n-1} +\omega_{n-2i-1}}$ has to vanish. \end{proof}

\section{Decomposition formulas: quartics}\label{section5}

\subsection{More formulas for tensor products}
First we shall need decomposition formulas for tensor products of some
Cartan powers of spin representations. We use the notation $\theta_i=
\epsilon_1+\cdots +\epsilon_i$; this is a fundamental weight $\omega_i$
when $i\le n-2$, but $\theta_{n-1}=\omega_{n-1}+\omega_n$ and $\theta_n=
2\omega_n$. We have
\begin{alignat*}{3}
& V_{3\omega_n}\otimes\Delta_+ =
\bigoplus_{j\ge 0}V_{\theta_n+\theta_{n-2j}}, \qquad &&
V_{3\omega_n}\otimes\Delta_- =
\bigoplus_{j\ge 0}V_{\theta_n+\theta_{n-2j-1}}, & \\
& V_{2\omega_n}\otimes V_{2\omega_{n-1}}= \bigoplus_{j,k\; {\rm odd}}
V_{\theta_{n-j}+\theta_{n-k}},\qquad &&
V_{2\omega_n}\otimes V_{\omega_{n-i}}= \bigoplus_{\substack{j+k\ge i
\\j+k-i\,{\rm even}}}V_{\theta_{n-j}+\theta_{n-k}},&
\end{alignat*}
Observe that, as follows from Weyl's dimension formula, for $p\ge q\ge 1$,
\begin{gather*}
\dim V_{\theta_n+\theta_{n-p}} = (p+1)^2
\frac{(2n)!(2n+1)!}{(n-p)!(n+p+2)!n!(n+1)!}, \\
\dim V_{\theta_{n-p}+\theta_{n-q}} = (p-q+1)(p+q+1)
\frac{(2n)!(2n+2)!}{(n-p)!(n+p+2)!(n-q+1)!(n+q+1)!}.
\end{gather*}

We will also need formulas for the tensor products of fundamental non spin
representations.
These are just wedge products of the vector representations, so one can
take their tensor products as $\fsl_{2n}$-modules, and then restrict to
$\fso_{2n}$-modules using the Littlewood restriction rules and their
generalization by King and Howe--Tan--Willenbring (see~\cite{htw}).
The result is the
following decomposition formula, for $p\ge q$:
\[
V_{\omega_{n-p}}\otimes V_{\omega_{n-q}}= \bigoplus_{
\substack{
p-q\le r-s\le p+q\le r+s \\ p+q-r-s\; {\rm even}}}
V_{\theta_{n-r}+\theta_{n-s}}
  \oplus 2\bigoplus_{\substack{p+q\le r+s\\ p+q-r-s\; {\rm even}}
}V_{\theta_{n-r}+\theta_{n-s}}.
\]

For future use we will need to understand in some detail the spin-equivariant
maps $V_{\omega_p}\otimes V_{\omega_q}\rightarrow
V_{\theta_r+\theta_s}$. Observe that we can def\/ine two basic maps
\begin{gather*}
 \alpha_{p,q}^{p-1,q+1}: \
 \Lambda^pV\otimes \Lambda^qV
 \rightarrow\Lambda^{p-1}V\otimes V\otimes \Lambda^qV
 \rightarrow\Lambda^{p-1}V\otimes \Lambda^{q+1}V,
\\
 \alpha_{p,q}^{p-1,q-1}: \
 \Lambda^pV\otimes \Lambda^qV
 \rightarrow \Lambda^{p-1}V\otimes V\otimes \Lambda^qV
 \rightarrow \Lambda^{p-1}V\otimes \Lambda^{q-1}V.
\end{gather*}

In the def\/inition of  $\alpha_{p,q}^{p-1,q+1}$ we used the natural map
$V\otimes \Lambda^qV\rightarrow \Lambda^{q+1}V$ def\/ined by the wedge product,
while in the def\/inition of  $\alpha_{p,q}^{p-1,q-1}$ we used the map
$V\otimes \Lambda^qV\rightarrow \Lambda^{q-1}V$ induced by the contraction
by the quadratic form on $V$.

A straightforward computation shows that
\[
\alpha_{p-1,q+1}^{p-2,q}\circ\alpha_{p,q}^{p-1,q+1}=
\alpha_{p-1,q-1}^{p-2,q}\circ\alpha_{p,q}^{p-1,q-1}.
\]
By such compositions, we can therefore def\/ine unambiguously, maps
\[
\alpha_{p,q}^{r,s}: \ \Lambda^pV\otimes \Lambda^qV \rightarrow
 \Lambda^rV\otimes \Lambda^sV
\]
for $p+q-r-s$ even and $|q-s|\le p-r$. Note that the latter condition is
equivalent to $p-q\ge r-s$ and $p+q\ge r+s$.
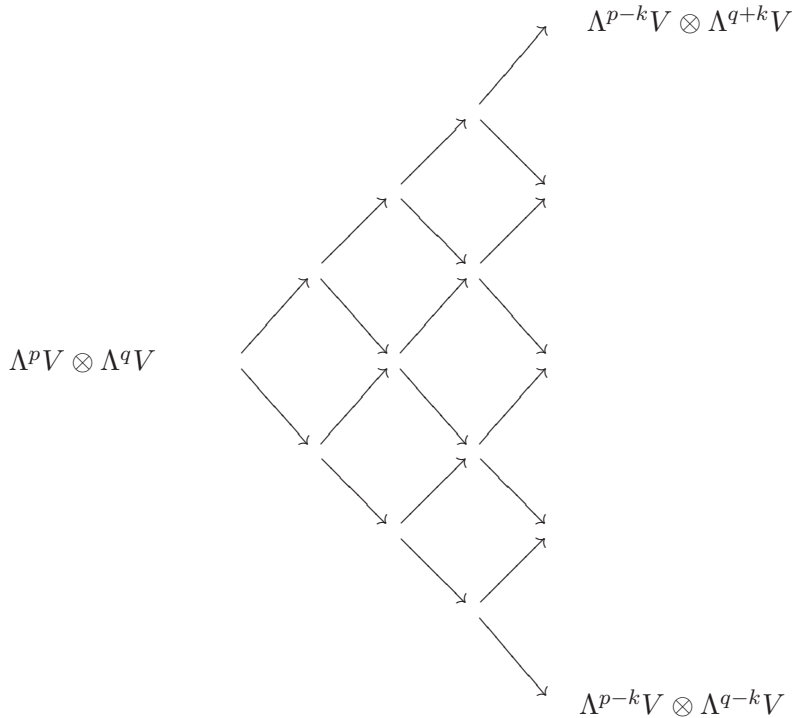
\[
\xymatrix{
 &  &  & & & &\hspace*{-6mm}\Lambda^{p-k}V\otimes \Lambda^{q+k}V\\
 &  & & &\ar[ur]\ar[dr] &&  \\
 & & & \ar[ur]\ar[dr] & &  &\\
 & &  \ar[ur]\ar[dr] & & \ar[ur]\ar[dr]& & \\
\Lambda^pV\otimes \Lambda^qV& \ar[ur]\ar[dr] &  & \ar[ur]\ar[dr] & && \\
 &  &  \ar[ur]\ar[dr] & & \ar[ur]\ar[dr]& & \\
 & & & \ar[ur]\ar[dr] & & & \\  & & & &\ar[ur]\ar[dr] & & \\
 &  &  & & & &\hspace*{-8mm}\Lambda^{p-k}V\otimes \Lambda^{q-k}V \\
}
\]

If we compose such a map $\alpha_{p,q}^{r,s}$ with the
projection to the Cartan component,
\[
\Lambda^pV\otimes \Lambda^qV \rightarrow\Lambda^rV\otimes \Lambda^sV
\rightarrow V_{\theta_r+\theta_s},
\]
we claim that the resulting map $\beta_{p,q}^{r,s}$ is non zero.
This is also true
if $s\ge n$, which can occur if $p+q\ge n$; in this case, since
$\Lambda^sV\simeq \Lambda^{2n-s}V$, the resulting map $\beta_{p,q}^{r,s}$
maps $\Lambda^pV\otimes \Lambda^qV$ to $V_{\theta_r+\theta_{2n-s}}$. We
claim that $\beta_{p,q}^{r,s}$ and $\beta_{p,q}^{r,2n-s}$ are independent,
and that this is the explanation for the multiplicities that are equal
to two in the decomposition formula above.

\subsection{Branching}

Second, we will use restriction formulas for representations
of type $V_{\theta_i+\theta_j}$, f\/irst from $D_{n+1}$ to~$B_n$,
then to $D_n$ (see again \cite{htw}).
We identify the weight $\theta_i+\theta_j$, for
$i\le j$, with the partition $(2^i1^{j-i})$, of length $j\le n+1$.
\begin{enumerate}\itemsep=0pt
\item To obtain the restriction of $V_{\theta_i+\theta_j}$ from
$D_{n+1}$ to $B_n$, take the sum of the
representations def\/ined by the partitions
\[
(2^i1^{j-i}),\quad (2^{i-1}1^{j-i+1}),\quad (2^i1^{j-i-1}),
\quad(2^{i-1}1^{j-i}),
\]
(where only those of length at most $n$ must be kept).
\item To obtain the restriction of $V_{\theta_i+\theta_j}$ from $B_n$
to $D_n$, make the same operation, except that if one of the
representation obtained is of the form $V_{\theta_n+\theta_i}=
V_{2\omega_n+\theta_i}$ (resp.\ $V_{2\theta_n}=V_{4\omega_n}$),
one has to add the mirror representation
$V_{2\omega_{n-1}+\theta_i}$ (resp. $V_{4\omega_{n-1}}$).
\end{enumerate}

\subsection{Decomposing quartics}\label{section5.3}

From the preceding formulas, and our decomposition for cubics,
we easily deduce the following statement:

\begin{proposition}
There exists integers $e^n_{i,j}$ and $f^n_i$ such that
the decomposition formula for the fourth symmetric power of a half-spin
representation is of the form
\[
S^4\Delta_+=\bigoplus_{0\le i,j\le n}e^n_{i,j}
V_{\theta_{n-i}+\theta_{n-j}}\oplus \bigoplus_{0\le i\le n}
f^n_iV_{2\omega_{n-1}+\theta_{n-i}}.
\]
Moreover $e^n_{i,j}=0$ if $i+j$ is odd and $f^n_i=0$ if $i$ is odd.
\end{proposition}

\begin{proof} Only the last assertion needs to be proved. Recall that the
weights of the half-spin representation, after a suitable choice of a maximal
torus of ${\rm Spin}(V)$, are of the form $\frac{1}{2}(\pm\epsilon_1\pm\cdots
\pm\epsilon_n)$, where the number of minus signs is even. In particular the
sum of the coef\/f\/icients is $\frac{n}{2}$ minus some even integer. This implies
that if $a_1\epsilon_1+\cdots+a_n\epsilon_n$ is any weight of $S^4\Delta_+$,
then $a_1+\cdots+a_n$ equals $2n$ mod $2$. In particular, if
$\theta_{n-i}+\theta_{n-j}$ is a weight of $S^4\Delta_+$, then $i+j$ must
be even. Similarly, if $2\omega_{n-1}+\theta_{n-i}$ is a weight of
$S^4\Delta_+$, then $i$ must be even. \end{proof}

We expect the same phenomena as for cubics, that is:

\medskip

 \noindent {\bf Conjecture.} {\it The integers
$e^n_{i,j}$ and $f^n_i$ are independent of $n$.}

\medskip

We have checked this conjecture up to $n=20$ with the help of the program
LiE \cite{LiE}.

The coef\/f\/icients $f_i=f^n_i$ are all equal to zero in that range,
except $f_{16}=f_{20}=1$.
The f\/irst coef\/f\/icients $e_{i,j}=e^n_{i,j}$ are given by the following table,
where they are displayed in such a way that the f\/irst line gives the
coef\/f\/icients $e_{0,j}$, and the diagonal gives the coef\/f\/icients $e_{i,i}$.
\[
\begin{tabular}{ccccccccccccccccccccc}
 1& 0& 0& 0& 1& 0& 1& 0& 2& 0& 1& 0& 3& 0& 2& 0& 4& 0& 3& 0& 5 \\
  & 0& 0& 0& 0& 0& 0& 0& 0& 1& 0& 1& 0& 2& 0& 2& 0& 3& 0& 4& 0  \\
  &  & 0& 0& 0& 0& 1& 0& 0& 0& 2& 0& 1& 0& 3& 0& 2& 0& 5& 0& 3 \\
  &  &  & 0& 0& 0& 0& 1& 0& 1& 0& 1& 0& 2& 0& 3& 0& 3& 0& 4& 0  \\
  &  &  &  & 1& 0& 0& 0& 1& 0& 1& 0& 3& 0& 1& 0& 4& 0& 3& 0& 6 \\
  &  &  &  &  & 0& 0& 0& 0& 1& 0& 0& 0& 2& 0& 2& 0& 3& 0& 3& 0  \\
  &  &  &  &  &  & 1& 0& 0& 0& 2& 0& 1& 0& 3& 0& 2& 0& 5& 0& 3 \\
  &  &  &  &  &  &  & 0& 0& 0& 0& 1& 0& 1& 0& 2& 0& 2& 0& 4& 0  \\
  &  &  &  &  &  &  &  & 1& 0& 0& 0& 2& 0& 1& 0& 4& 0& 2& 0& 5 \\
  &  &  &  &  &  &  &  &  & 1& 0& 0& 0& 2& 0& 1& 0& 3& 0& 3& 0  \\
  &  &  &  &  &  &  &  &  &  & 1& 0& 0& 0& 2& 0& 1& 0& 4& 0& 2 \\
  &  &  &  &  &  &  &  &  &  &  & 0& 0& 0& 0& 2& 0& 1& 0& 3& 0  \\
  &  &  &  &  &  &  &  &  &  &  &  & 2& 0& 0& 0& 3& 0& 2& 0& 5 \\
  &  &  &  &  &  &  &  &  &  &  &  &  & 1& 0& 0& 0& 2& 0& 1& 0  \\
  &  &  &  &  &  &  &  &  &  &  &  &  &  & 1& 0& 0& 0& 3& 0& 1 \\
  &  &  &  &  &  &  &  &  &  &  &  &  &  &  & 1& 0& 0& 0& 3& 0  \\
  &  &  &  &  &  &  &  &  &  &  &  &  &  &  &  & 2& 0& 0& 0& 3 \\
  &  &  &  &  &  &  &  &  &  &  &  &  &  &  &  &  & 1& 0& 0& 0  \\
  &  &  &  &  &  &  &  &  &  &  &  &  &  &  &  &  &  & 2& 0& 0 \\
  &  &  &  &  &  &  &  &  &  &  &  &  &  &  &  &  &  &  & 1& 0 \\
  &  &  &  &  &  &  &  &  &  &  &  &  &  &  &  &  &  &  &  & 2
\end{tabular}
\]

We can try to prove this conjecture in the same way as we proved the
similar statement for cubics, that is, by induction on $n$.
So we restrict the above formula
for $S^4\Delta_+$ to $D_{n-1}$ and deduce an inductive relation for the
multiplicities.

To obtain this equation we need a formula for the tensor products
of $S^3\Delta_+\otimes\Delta_-$. One can check that for $p\ge 2$,
\begin{gather*}
V_{\omega_n+\omega_{n-p}}\otimes\Delta_+ =  \bigoplus_{
\substack{r\ge p\ge s\\ r+s-p\; {\rm even}}}V_{\theta_{n-r}+\theta_{n-s}}, \qquad
V_{\omega_n+\omega_{n-p}}\otimes\Delta_- =  \bigoplus_{
\substack{r\ge p\ge s\\ r+s-p\; {\rm odd}}}V_{\theta_{n-r}+\theta_{n-s}}.
\end{gather*}

One can then write down some inductive relations for the multiplicities
$e^n_{i,j}$ and $f^n_i$, which unfortunately are not suf\/f\/icient to
compute them all (contrary to the cubic case), and a fortiori not
suf\/f\/icient to prove the conjecture: some additional idea is needed.

Admitting the conjecture, we obtain the following recursive relations
for the multiplicities:
\begin{gather*}
e_{i-1,i-1}+e_{i,i}+e_{i+1,i+1} = \lfloor\frac{i}{4}\rfloor
+\delta_{i,0}+\delta_{i,odd}, \\
e'_{0,2i}+e_{1,2i-1}+e_{1,2i+1} = \binom{\lfloor\frac{i}{2}\rfloor+1}{2}, \\
e_{i-1,j}+e_{i,j-1}+e_{i+1,j}+e_{i,j+1} = \sum_{i\le 2p\le j}a_p
+ \sum_{i\le 2p+1\le j}b_p,
\end{gather*}
where we have let $e'_{0,j}=e_{0,j}+f_j$. Moreover the third type of equations
holds for $j>i>0$, with the caveat that for $i=1$, $e_{0,j}$ has to
be replaced by $e'_{0,j}$.

The f\/irst series of equations allow to compute the diagonal coef\/f\/icients
$e_{i,i}$. Their generating series is
\[
\sum_{i\ge 0}e_{i,i}x^i=\frac{1+x^9}{(1-x^4)(1-x^6)}.
\]
We conjecture that the generating series of the multiplicities $e_{0,i}$
and $f_i$ are
\[
\sum_{i\ge 0}e_{0,i}x^i=\frac{1}{(1-x^4)(1-x^6)(1-x^8)}, \qquad
\sum_{i\ge 0}f_{i}x^i=\frac{x^{16}}{(1-x^4)(1-x^6)(1-x^8)}.
\]
Once the coef\/f\/icients $e_{i,i}$, $e_{0,i}$ and $f_{i}$ are known,
the recursive relations above allow to compute all the multiplicities
$e_{i,j}$, and to write down their generating series as an explicit rational
function.

\section{Quartic equations of the secant variety}

We f\/irst make the following easy observation (see \cite{LMsec}):
the space $I_4(\sigma(S_+))$
of quartic equations of the secant variety to the spinor variety, is the
orthogonal in~$S^4\Delta_+^\vee$ to the subspace of~$S^4\Delta_+$ generated
by the fourth powers $(a+b)^4$, where $a$, $b$ are pure spinors (i.e., belong
to the cone over~$S_+$). Such a tensor decomposes into homogeneous components
that can be treated separately: the fourth powers $a^4$ generate the
Cartan components $V_{4\omega_n}$ of $S^4\Delta_+$; the terms $a^3b$
generate the image of $V_{3\omega_n}\otimes V_{\omega_n}\subset
S^3\Delta_+\otimes\Delta_+$ in  $S^4\Delta_+$; the terms $a^2b^2$
generate the image of $V_{2\omega_n}\otimes V_{2\omega_n}$ (or rather
its symmetric part). These
tensor products are known: we have
\begin{gather*}
V_{3\omega_n}\otimes V_{\omega_n}
 = \bigoplus_{p}V_{\theta_{n}+\theta_{n-2p}}, \qquad
S^2V_{2\omega_n} = \bigoplus_{p-q \;{\rm even}}V_{\theta_{n-2p}+\theta_{n-2q}}.
\end{gather*}

The only components that appear in both decomposition
are $V_{4\omega_n}$ and the $V_{\theta_{n}+\theta_{n-2p}}$'s,
for $p$ an even integer. This implies that
\[
\CC[\sigma (S_+)]_4\subset
\bigoplus_{\substack{p\ge q>0\\p-q \; {\rm even}}}
V_{\theta_{n-2p}+\theta_{n-2q}}^\vee
\oplus 2\bigoplus_{p\; {\rm even}}
V_{\theta_{n}+\theta_{n-2p}}^\vee
\oplus \bigoplus_{p\; {\rm odd}}
V_{\theta_{n}+\theta_{n-2p}}^\vee.
\]

\begin{theorem}\label{quartics}
\[
\CC[\sigma (S_+)]_4=
\bigoplus_{\substack{p\ge q\ge 1\\p-q \;{\rm even}\\ (p,q)\ne (1,1)}}
V_{\theta_{n-2p}+\theta_{n-2q}}^\vee
\oplus \bigoplus_{p\ne 1}
V_{\theta_{n}+\theta_{n-2p}}^\vee.
\]
\end{theorem}

\begin{proof}
There are two things to prove. First, that the
$V_{\theta_{n}+\theta_{n-2p}}^\vee$'s, for $p$ even, appear in
$\CC[\sigma (S_+)]_4$ with multiplicity one.
Second, that the $V_{\theta_{n-2p}+\theta_{n-2q}}^\vee$'s,
for $p-q$ even, or $q=0$ and $p$ odd, have non zero multiplicity.

For the f\/irst assertion, recall that $V_{\theta_{n-2p}}$ is, for
$p$ even, a component of $S^2\Delta_+$, with multiplicity one.
Let $V=\sum_iv_i\otimes v'_i$ be a generator of the corresponding highest
weight line in  $S^2\Delta_+\subset\Delta_+\otimes \Delta_+$.
Denote by $u$ a highest weight vector of $\Delta_+$. Then
$\sum_i(uv_i)\otimes (uv'_i)$
is a highest weight vector in $S^2V_{2\omega_n}$, while
$\sum_i(u^2v_i)\otimes v'_i$
is a highest weight vector in $V_{3\omega_n}\otimes V_{\omega_n}$.
Since their images in $S^4V_{2\omega_n}$ are both equal to
$\sum_iu^2v_iv'_i$, we deduce that the components $V_{\theta_n+\theta_{n-2p}}$
inside $S^2V_{2\omega_n}$ and inside $V_{3\omega_n}\otimes V_{\omega_n}$
generate a unique copy of $V_{\theta_n+\theta_{n-2p}}$
inside $S^4V_{2\omega_n}$. This proves the f\/irst claim.

For the second assertion, consider a component
$V_{\theta_{n-2p}+\theta_{n-2q}}$ coming from $S^2V_{2\omega_n}$,
with $p\ge q$.
We want to check that the tensors of the form $e^2f^2$ in
$S^4V_{2\omega_n}$, for $e$ and $f$ in the cone over the spinor
variety, do generate such a component.

Suppose f\/irst that $p$ and $q$ are both even. Consider the map
\[
S^2V_{2\omega_n}\hookrightarrow S^2V_{\omega_n}\otimes S^2V_{\omega_n}
\stackrel{\beta_{n-2p}\otimes\beta_{n-2q}}{\lra}
V_{\theta_{n-2p}}\otimes V_{\theta_{n-2q}}.
\]
Since $\beta_{n-2p}(e,e)=0$ (at least for $p\ge 1$), the tensor $e^2f^2$
is mapped by this morphism to
\[
\beta_{n-2p}(e,f)\otimes \beta_{n-2q}(e,f) \in
V_{\theta_{n-2p}}\otimes V_{\theta_{n-2q}}.
\]
To check that $\beta_{n-2p}(e,f)\otimes \beta_{n-2q}(e,f)$ has a non zero
projection to the Cartan component $V_{\theta_{n-2p}+\theta_{n-2q}}$,
we just need to pair it with highest weight vectors of that component
(which is self-dual). Such highest weight vectors are of the form
\[
g_1\wedge\cdots\wedge g_{n-2p}\otimes g_1\wedge\cdots\wedge g_{n-2q},
\]
where the vectors $g_1, \ldots ,g_{n-2q}$ generate an isotropic subspace
of $V$ (recall that we have supposed $p\ge q$). So it suf\/f\/ices to check
that $\beta_{n-2p}(e,f)$ pairs non trivially with a generic tensor
$g_1\wedge\cdots\wedge g_{n-2p}$. But this is obvious, since such tensors
generate the whole space $V_{\theta_{n-2p}}$.

Now suppose that $p$ and $q$ are both odd. Then we cannot  use the same
morphism as before, and instead we use the composition
\[
S^2V_{2\omega_n}\hookrightarrow S^2V_{\omega_n}\otimes S^2V_{\omega_n}
\stackrel{\beta_{n-2p+2}\otimes\beta_{n-2q-2}}{\lra}
V_{\theta_{n-2p+2}}\otimes V_{\theta_{n-2q-2}}\lra
V_{\theta_{n-2p}}\otimes V_{\theta_{n-2q}}.
\]
We conclude as in the previous case, except when $p=1$ for which we get
zero. By symmetry we can always get the component
$V_{\theta_{n-2p}+\theta_{n-2q}}$, except in the case where $(p,q)=(1,1)$.
\end{proof}

Comparing with the decomposition of $S^4\Delta_+^\vee$ we can of course
deduce the decomposition of $I_4(\sigma(S_+))$ into irreducible components.
It has the same ``odd'' part (by this we mean the same multiplicities
on the components of type $V_{2\omega_{n-1}+\theta_{n-j}}$), while
the multiplicities of components $V_{\theta_{n-i}+\theta_{n-j}}$ in
the ``even'' part are given, for low values of $i$ and $j$, by the
following table:
\[
\begin{tabular}{ccccccccccccccccccccc}
 0& 0& 0& 0& 0& 0& 0& 0& 1& 0& 0& 0& 2& 0& 1& 0& 3& 0& 2& 0& 4 \\
  & 0& 0& 0& 0& 0& 0& 0& 0& 1& 0& 1& 0& 2& 0& 2& 0& 3& 0& 4& 0  \\
  &  & 0& 0& 0& 0& 0& 0& 0& 0& 1& 0& 1& 0& 2& 0& 2& 0& 4& 0& 3 \\
  &  &  & 0& 0& 0& 0& 1& 0& 1& 0& 1& 0& 2& 0& 3& 0& 3& 0& 4& 0  \\
  &  &  &  & 0& 0& 0& 0& 0& 0& 1& 0& 2& 0& 1& 0& 3& 0& 3& 0& 5 \\
  &  &  &  &  & 0& 0& 0& 0& 1& 0& 0& 0& 2& 0& 2& 0& 3& 0& 3& 0  \\
  &  &  &  &  &  & 0& 0& 0& 0& 1& 0& 1& 0& 2& 0& 2& 0& 4& 0& 3 \\
  &  &  &  &  &  &  & 0& 0& 0& 0& 1& 0& 1& 0& 2& 0& 2& 0& 4& 0  \\
  &  &  &  &  &  &  &  & 0& 0& 0& 0& 1& 0& 1& 0& 3& 0& 2& 0& 4 \\
  &  &  &  &  &  &  &  &  & 1& 0& 0& 0& 2& 0& 1& 0& 3& 0& 3& 0  \\
  &  &  &  &  &  &  &  &  &  & 0& 0& 0& 0& 1& 0& 1& 0& 3& 0& 2 \\
  &  &  &  &  &  &  &  &  &  &  & 0& 0& 0& 0& 2& 0& 1& 0& 3& 0  \\
  &  &  &  &  &  &  &  &  &  &  &  & 1& 0& 0& 0& 2& 0& 2& 0& 4 \\
  &  &  &  &  &  &  &  &  &  &  &  &  & 1& 0& 0& 0& 2& 0& 1& 0  \\
  &  &  &  &  &  &  &  &  &  &  &  &  &  & 0& 0& 0& 0& 2& 0& 1 \\
  &  &  &  &  &  &  &  &  &  &  &  &  &  &  & 1& 0& 0& 0& 3& 0  \\
  &  &  &  &  &  &  &  &  &  &  &  &  &  &  &  & 1& 0& 0& 0& 2 \\
  &  &  &  &  &  &  &  &  &  &  &  &  &  &  &  &  & 1& 0& 0& 0  \\
  &  &  &  &  &  &  &  &  &  &  &  &  &  &  &  &  &  & 1& 0& 0 \\
  &  &  &  &  &  &  &  &  &  &  &  &  &  &  &  &  &  &  & 1& 0 \\
  &  &  &  &  &  &  &  &  &  &  &  &  &  &  &  &  &  &  &  & 1
\end{tabular}
\]

Of course these multiplicities will be independent of $n$ if and only if
the conjecture in Section~\ref{section5.3} is correct.

For example, we get (denoting $S_n$ the spinor variety $S_+$ in type $D_n$):
\begin{gather*}
 I_4(\sigma (S_7))=V_{\omega_4}, \qquad
 I_4(\sigma (S_8))=V_{\omega_1+\omega_5}\oplus V_{2\omega_8}, \\
 I_4(\sigma (S_9))=V_{\omega_2+\omega_6}\oplus V_{\omega_1+2\omega_9}
\oplus V_{\omega_8}\oplus V_{\omega_6}\oplus V_{\omega_4}\oplus V_{\omega_0}.
\end{gather*}
The f\/irst of these statements is contained in \cite[Theorem 1.3]{LW2},  where it is proved
that $V_{\omega_4}$ generates the full ideal $I(\sigma (S_7))$.

Comparing the previous theorem with Theorem~3.11 in \cite{LW1}, we deduce:

\begin{corollary}
The kernel of the restriction map $\CC[\sigma(S_+)]_4\ra \CC[\tau(S_+)]_4$
is equal to
\[
\bigoplus_{\substack{p\ge q\ge 1\\p-q \;{\rm even}\\  p+q>4}}
V_{\theta_{n-2p}+\theta_{n-2q}}^\vee
\oplus \bigoplus_{p>4}
V_{\theta_{n}+\theta_{n-2p}}^\vee.
\]
\end{corollary}

We also remark that quartic equations are not always induced by
cubic equations.

\begin{proposition}
There is a component $V_{\omega_{n-3}+\omega_{n-7}}^\vee$ inside
$I_4(\sigma(S_+))$, of multiplicity one, consisting in quartic equations
which are not induced by cubics -- hence generators of $I(\sigma(S_+))$.
\end{proposition}

\begin{proof}
The tensor product formulas given in Section~\ref{section5} imply that the tensor
product of $V_{\omega_{n}{+}\omega_{n{-}2p}}$ or $V_{\omega_{n-1}+\omega_{n-2p-1}}$
by $\Delta_+$, for $p\ge 4$, does not contain any copy of
$V_{\omega_{n-3}+\omega_{n-7}}$. Therefore the tensor product of
 $I_3(\sigma(S_+))$ by $\Delta_+^\vee$ cannot contain
$V_{\omega_{n-3}+\omega_{n-7}}^\vee$, which implies our claim. \end{proof}

\section{Freudenthal varieties}

The spinor varieties belong to the family of Freudenthal varieties
\[
LG\subset \PP(\wedge^{\langle n\rangle}\CC^{2n}),
\qquad G\subset \PP(\wedge^n\CC^{2n}),
\qquad S_+\subset \PP(\Delta_+)
\]
which share, especially for $n=3$, many remarkable properties \cite{ky,LM}.
Here we denoted $G=G(n,2n)$ the usual
Grassmannian,
and $LG=LG(n,2n)$ the Lagrangian Grassmannian, in their respective
Pl\"ucker embeddings. Moreover we restrict to the spinor varieties
of type $D_{2n}$. The varieties $LG$, $G$, $S_+$ are then Hermitian symmetric
spaces of the same rank $n$. In fact they can be
considered as models of the same variety over the (complexif\/ied)
normed algebras~$\RR$,~$\CC$ and~$\HH$. For $n=3$ there is even an
exceptional Freudenthal variety ${\mathbf G}$ over the Cayley algebra~$\OO$,
which is the unique compact Hermitian symmetric space  of type~$E_7$.

Every Freudenthal variety is easily seen to be 3-transitive, and this leads
to analogs of the homogeneous spaces embeddings discussed after Proposition
\ref{3-tr}:
\begin{alignat*}{4}
 & Sp_n\hookrightarrow LG,  \qquad && Sp_{2n}/GL_n\hookrightarrow LG\times LG,
 \qquad && Sp_{2n}/O_n\hookrightarrow LG\times LG\times LG, & \\
 & GL_n\hookrightarrow G, \qquad && GL_{2n}/GL_n\times GL_n\hookrightarrow G\times G,
 \qquad && GL_{2n}/GL_n\hookrightarrow G\times G\times G, & \\
 & SO_n\hookrightarrow S_+, \qquad && SO_{2n}/GL_n\hookrightarrow S_+\times S_+,
 \qquad && SO_{2n}/Sp_n\hookrightarrow S_+\times S_+\times S_+, & \\
&  E_6\hookrightarrow {\mathbf G}, \qquad && E_7/\CC^*E_6\hookrightarrow {\mathbf G}\times  {\mathbf G},
 \qquad && E_7/F_4\hookrightarrow {\mathbf G}\times  {\mathbf G}\times  {\mathbf G}.&
\end{alignat*}
Of course this is strongly reminiscent of Freudenthal's magic square and
its higher rank generalizations~\cite{LM}.

Cubic equations of the secant variety of a Freudenthal variety $F\subset
\PP V_{\omega}$
can be described uniformly. In fact the decomposition of $S^3V_{\omega}$ is
known, and there are two types of isotypic components:
\begin{enumerate}\itemsep=0pt
\item Those whose highest weight does not appear among the weights
of $V_{2\omega}\otimes V_{\omega}$; as we have already noticed,
the duals of these components must belong to $I_3(\sigma(F))$.
\item Those whose highest weight does appear among the weights
of $V_{2\omega}\otimes V_{\omega}$. In fact this tensor product turns
out to be multiplicity
free; we write it down as
\[
V_{2\omega}\otimes V_{\omega}=\bigoplus_iV_{\omega+\Omega_i}.
\]
\end{enumerate}

Let us discuss the case of $G=G(n,2n)$, the usual Grassmannian.
Here $U^\vee=\wedge^n\CC^{2n}$, and the decomposition of $S^3U$
follows from the computations of Chen, Garsia and Remmel \cite{CGR}.
In fact these authors
compute the plethysm $S^3(S^n)$ rather than $S^3(\wedge^n)$,
but its is known that $S^3(\wedge^n)$ is dual to $S^3(S^n)$ for $n$
even; for $n$ odd $S^3(\wedge^n)$ is dual to $\wedge^3(S^n)$, which is also
computed in \cite{CGR}. Here by ``dual'', we mean that the
highest weights of the irreducible components are coded by dual
partitions, with the same multiplicities. This yields:
\[
 S^3(\wedge^n)= \bigoplus_{\substack{a,b\le n\\ a\le 2b, \; b\le 2a}}
m_{a,b}V_{\omega_{n-a}+\omega_{n+a-b}+\omega_{n+b}},
\]
where the multiplicity $m_{a,b}$ is given by the following rule.
If $b\ge a$, then $m_{a,b}=\lfloor\frac{2a-b+1}{6}\rfloor$ if $2a-b=1$ mod $6$
or  $a$ and $b$ are both even.
Otherwise $m_{a,b}=\lceil\frac{2a-b+1}{6}\rceil$.
If $a>b$,  then $m_{a,b}=\lfloor\frac{2b-a+1}{6}\rfloor$ if $2b-a=1$ mod $6$
or  $a$ and $b$ are both even. Otherwise $m_{a,b}=\lceil\frac{2b-a+1}{6}\rceil$.

Observe that the decomposition of $S^3(\wedge^n)$ is thus notably more
complicated that the decomposition of $S^3\Delta_+$. But most components,
more precisely all those for which $a\ne b$,
will automatically vanish on $\sigma(G)$, since
\[
(\wedge^n)^{(2)}\otimes \wedge^n =
\bigoplus_{a\le n}V_{\omega_{n-a}+\omega_{n}+\omega_{n+a}}.
\]
This indicates that we should pay special attention to the multiplicities
$m_a:=m_{a,a}$, which are given by the formula
\[
m_{6r+s}=r+1-\delta_{s,1}, \qquad r\ge 0,\quad 0\le s\le 5.
\]
Let us compute the generating series of these multiplicities:
\begin{gather*}
\sum_{k\ge 0}m_kx^k = \big(1+x+x^2+x^3+x^4+x^5\big)\sum_{r\ge 0}(r+1)x^{6r}
  -\sum_{r\ge 0}x^{6r+1} \\
 \hphantom{\sum_{k\ge 0}m_kx^k}{} = \big(1+x+x^2+x^3+x^4+x^5\big)\big(1-x^6\big)^{-2}-x\big(1-x^6\big)^{-1} \\
\hphantom{\sum_{k\ge 0}m_kx^k}{}= \frac{1}{1-x^6}\left(\frac{1}{1-x}-x\right)=\frac{1}{1-x^6}\frac{1-x+x^2}{1-x} \\
\hphantom{\sum_{k\ge 0}m_kx^k}{}= \frac{1}{1-x^6}\frac{1+x^3}{1-x^2}=\frac{1}{(1-x^2)(1-x^3)}.
\end{gather*}
This is the rational function we already met in Theorem \ref{cubic-dec}!

One can check that the same phenomenon holds for the Lagrangian Grassmannian.
We f\/inally get the uniform statement:

\begin{theorem}\label{theorem5}
Let $X\subset\PP(V_\omega)$ be a Freudenthal variety of rank $n$.
Then there exist dominant weights $\Omega_0=2\omega,\Omega_1, \ldots ,\Omega_n$
such that:
\begin{enumerate}\itemsep=0pt
\item[$1.$] for any $k\le l$, one has
\[
V_{k\omega}\otimes V_{l\omega}=\bigoplus_{0\le i_1\le\cdots\le i_k\le n}
V_{(l-k)\omega+\Omega_{i_1}+\cdots+\Omega_{i_k}},
\]
in particular
\[
V_{\omega}\otimes V_\omega=\bigoplus_{i=0}^nV_{\Omega_i},\qquad
V_{2\omega}\otimes V_\omega=\bigoplus_{i=0}^nV_{\omega+\Omega_i};
\]
\item[$2.$] the multiplicity $m_i$ of $V_{\omega+\Omega_i}$ inside $S^3V_\omega$
is given by the generating series
\[
\sum_{i\ge 0}m_ix^i=\frac{1}{(1-x^2)(1-x^3)};
\]
\item[$3.$] the degree three part of $\CC[\sigma(X)]$ is
\[
\CC[\sigma(X)]_3=\bigoplus_{i\ne 1}V_{\omega+\Omega_i}^\vee.
\]
\end{enumerate}
\end{theorem}

Explicitly, the weights $\Omega_i$ are the following:
\[
\Omega_i = \begin{cases}
2\omega_{n-i} & \text{in type $C_n$}, \\
\omega_{n-i}+\omega_{n+i} & \text{in type $A_{2n-1}$}, \\
\theta_{2n-2i} & \text{in type $D_{2n}$}.
\end{cases}
\]

\section{A toy case}

It may be interesting to observe that the list of Freudenthal varieties can
be completed by two ``degenerate'' cases,
\[
V=v_n(\PP^1)\subset \PP(S^n\CC^{2}),
\qquad P=(\PP^1)^n\subset \PP((\CC^{2})^{\otimes n}).
\]
The complete series $V$, $P$, $LG$, $V$, $S_+$, for a given $n$, has in particular the nice property
that each variety is a (special) linear section of the next variety in the series.

The case of the rational normal curve $V$ is the simplest one. Its secant
varieties are well understood (see e.g.~\cite[Part I, 1.3]{iaka}).
In particular the f\/irst secant variety
has a nice desingularization in terms of the {\it secant bundle}.
For simplicity
let $U=\CC^{2}$. On the projective plane $\PP^2=\PP(S^2U)$, the secant bundle is the rank two
vector bundle $E$ def\/ined by the exact sequence
\begin{gather}\label{eq1}
0\ra S^{n-2}U^\vee\otimes\cO_{\PP^2}(-1)\ra S^{n}U^\vee\otimes\cO_{\PP^2}\ra E\ra 0.
\end{gather}
Obviously $E$ is generated by global sections and $H^0(\PP^2, E)=S^{n}U^\vee$. The line bundle
$\cO_E(1)$ def\/ines a morphism $\eta$ from the threefold $\PP(E)$ to $\PP (H^0(\PP^2, E)^\vee)=\PP(S^nU)$.
The next lemma is contained in \cite[Section 3]{kanev}.

\begin{lemma}
The map $\eta : \PP(E)\ra\PP(S^nU)$ is a desingularization of $\sigma(V)$. It induces
an isomorphism
\[
\CC [\sigma(V)]_k\simeq H^0(\PP(E),\cO_E(k))=H^0(\PP^2, S^kE).
\]
\end{lemma}

Since $S^kE$ has no higher cohomology, as easily follows from \eqref{eq1},
we can compute
$H^0(\PP^2, S^kE)$ $=\chi(\PP^2, S^kE)$ as a (virtual) $GL(U)$-module, by using the equivariant
localization formula of Atiyah--Bott~\cite{AB}. In order to state the result we need a few more
notations. We consider a maximal torus $T$ in $GL(U)$, and a compatible basis
$(e,f)$ of $U$. We denote by $x$, $y$ the corresponding characters of $T$. For $g\in T$,
the formula reads
\begin{gather*}
{\rm Trace}(g,\chi(\PP^2, S^kE))=\sum_{p\in (\PP^2)^T}\frac{{\rm Trace}(g,S^kE_p)}{\det(1-g^{-1},T_p\PP^2)}.
\end{gather*}
Here $(\PP^2)^T$ denotes the set of f\/ixed points of the action of $T$ on $\PP^2$,
which are the three lines in $S^2U$ generated by $e^2$, $ef$, $f^2$. In order to get a compact
formula, we introduce an indeterminate $t$ and def\/ine
\begin{gather*}
P_n(t,x,y)=\sum_{k\ge 0}t^k\, {\rm Trace}(g,\chi(\PP^2, S^kE)).
\end{gather*}
Let us compute this in from the Atiyah--Bott localization formula. At the f\/ixed point
$p=[e^2]$, the f\/iber of $E$ is the space of degree $n$ polynomials modded out by those
divisible by $e^2$. A $T$-eigenbasis of $E_p$ is thus given by the images of $f^n,ef^{n-1}$.
The induced $T$-eigenbasis of $S^kE_p$ consists in the degree $k$ monomials in $e$, $f$, multiplied
by $f^{k(n-1)}$. Moreover, the $T$-module $T_p\PP^2={\rm Hom}(\langle e^2\rangle, \langle ef,f^2\rangle)$
has weights $y/x$ and $y^2/x^2$. The same straightforward analysis at the two other f\/ixed points
yields the formula
\begin{gather*}
{\rm Trace}(g,\chi(\PP^2, S^kE))=\frac{y^{k(n-1)}h_k(x,y)}{(1-\frac{x}{y})(1-\frac{x^2}{y^2})}
+\frac{x^{k(n-1)}h_k(x,y)}{(1-\frac{y}{x})(1-\frac{y^2}{x^2})}
+\frac{h_k(x^n,y^n)}{(1-\frac{x}{y})(1-\frac{y}{x})},
\end{gather*}
where $h_k(x,y)$ denotes the sum of degree $k$ monomials in $x$, $y$. Summing over $k$, we get
the closed expression $P_n(t,x,y)=P'_n(t,x,y)/(x-y)(x^2-y^2)$, where
\[
P'_n(t,x,y)=\frac{x^3}{(1-tx^n)(1-tx^{n-1}y)}+\frac{y^3}{(1-ty^n)(1-txy^{n-1})}
-\frac{xy(x+y)}{(1-tx^n)(1-ty^n)}.
\]
Extracting the coef\/f\/icient of $t^3$, we get the character $P_n^3(x,y)=Q_n^3(x,y)/(x-y)$
of the $GL(U)$-module $H^0(\PP^2, S^3E))=\CC [\sigma(V)]_3$, with
\[
Q_n^3(x,y)=x^{2n}(x^{n+1}+x^{n-1}y^2+\cdots +xy^n)-y^{2n}(x^ny+\cdots +x^2y^{n-1}+y^{n+1}).
\]
Since the character of the $GL(U)$-module $S^pU\otimes (\det U)^q$ is $(x^{p+q+1}y^q-x^qy^{p+q+1})/(x-y)$,
this means that
\[
\CC [\sigma(V)]_3 = S^{3n}U\oplus\bigoplus_{k=2}^{n}S^{3n-2k}U\otimes (\det U)^k.
\]
This is in agreement with Theorem~\ref{theorem5}, for $\omega=n\omega_1$, but taken with a grain of salt.
Indeed, all the formulas in Theorem~\ref{theorem5} are correct, with
\[
\Omega_i=2\omega-2i\omega_1,
\]
except for the f\/irst one for $k\ge 2$.  Indeed, the weights $\Omega_i$, in the special case
we are dealing with, are not linearly independent (recall that this is a degenerate case!).
To get a correct result from the f\/irst formula of Theorem~\ref{theorem5} (which is just an instance of the
Clebsch--Gordan formula), one needs to correct this by taking all non zero multiplicities to the
value one.

  One can go further and compute the character $P_n^k(x,y)=Q_n^k(x,y)/(x-y)$
of the $GL(U)$-module $H^0(\PP^2, S^kE)=\CC [\sigma(V)]_k$. Indeed, if $q_{n,s}^k$ is the
coef\/f\/icient of $x^{nk+1-s}y^s$ in this polynomial, one can deduce for $nk\ge 2s$ the induction
formula
\[
q_{n,s}^k-q_{n,s-2}^k=\delta_{s\le k}-\delta_{n|s-1}-\delta_{n|s-2}.
\]
This leads to the following statement:

\begin{proposition}
Let $\epsilon\in\{0,1\}$ be the parity of $s$, with $nk\ge 2s$.
The multiplicity of $S^{kn-2s}U\otimes (\det U)^s$ inside $\CC [\sigma(V)]_k$ is
\[
q_{n,s}^k=\min\left( \lfloor\frac{s}{2}\rfloor,\lfloor\frac{k-\epsilon}{2}\rfloor\right)-\lfloor\frac{s-1}{n}\rfloor.
\]
\end{proposition}

For example, for $k=4$ one gets that the multiplicity $q_{n,s}^4$ is equal to
$0$ for $s=1$, to $2$ for $s$ even and $4\le s\le 2n$, and to $1$ otherwise.
For $n=3$ the secant variety is the whole of $\PP^3$, and one recovers the
formulas of Chen--Garsia--Remmel for the symmetric powers of $S^3\CC^2$.

For the tangent surface $\tau(V)$ to the normal rational curve we have
\[
\CC[\tau(V)]_k=\bigoplus_{\substack{s\le k \\ s\ne 1}}S^{kn-2s}U\otimes (\det U)^s.
\]
We recover the fact that we have already observed for the spinor varieties: $\tau(V)$ has,
for $n\ge 4$, and contrary to $\sigma(V)$, some non trivial cubic equations.

It is tempting to imagine that very similar formulas could hold for the
other types
of Freudenthal varieties. Observe, nevertheless, that the ideal of
$\sigma(V)$
is generated by cubics (obtained as $3\times 3$ minors of catalecticant
matrices \cite{kanev}).
We have seen that this statement is not true for spinor varieties.

\subsection*{Acknowledgments} I thank J.M.~Landsberg, G.~Ottaviani
and J.~Weyman for useful discussions.

\pdfbookmark[1]{References}{ref}
\LastPageEnding

\end{document}